\documentclass[11pt]{amsart}
\usepackage{amsfonts}
\usepackage{amsmath}
\usepackage{amsthm}
\usepackage{url}
\usepackage[all]{xy}
\usepackage{latexsym}
\usepackage{amssymb}
\usepackage[cp850]{inputenc}
\usepackage{amsfonts}
\usepackage{graphicx}
\usepackage{dsfont}
\usepackage{float}

\usepackage[usenames]{color}

\newtheorem{sat}{Theorem}[section]		
\newtheorem{lem}[sat]{Lemma}
\newtheorem{kor}[sat]{Corollary}			
\newtheorem{prop}[sat]{Proposition}

\newtheorem*{defi*}{Definition}			
\newtheorem*{bei*}{Example}
\newtheorem*{sat*}{Theorem}				
\newtheorem*{kor*}{Corollary}
\newtheorem*{rmk*}{Remark}				
	
\newtheorem*{quest*}{Question}

\let\ssection=\section
\renewcommand{\section}{\setcounter{equation}{0}\ssection}

\newtheorem*{namedtheorem}{\theoremname}
\newcommand{\theoremname}{testing}

\theoremstyle{remark}
\newtheorem*{bem}{Remark}

\newtheorem*{namedtheoremr}{\theoremnamer}
\newcommand{\theoremnamer}{testing}

			\newcommand{\BH}{\mathbb H}
\newcommand{\BR}{\mathbb R}			
			
\newcommand{\BS}{\mathbb S}			\newcommand{\BZ}{\mathbb Z}
\newcommand{\BF}{\mathbb F}

		\newcommand{\CF}{\mathcal F}
\newcommand{\CG}{\mathcal G}		\newcommand{\CH}{\mathcal H}

\newcommand{\CS}{\mathcal S}

\newcommand{\actson}{\curvearrowright}
\newcommand{\D}{\partial}

\DeclareMathOperator{\Diff}{Diff}	%	Diffeomorphimen einer Mf
\DeclareMathOperator{\Id}{Id}		%	Identit\"at
\DeclareMathOperator{\Ker}{Ker}

\newcommand{\comment}[1]{}

\DeclareMathOperator{\Stab}{Stab}

\newcommand{\fsubd}{\mathrel{{\scriptstyle\searrow}\kern-1ex^d\kern0.5ex}}
\newcommand{\bsubd}{\mathrel{{\scriptstyle\swarrow}\kern-1.6ex^d\kern0.8ex}}

\renewcommand{\epsilon}{\varepsilon}
\renewcommand{\le}{\leqslant}
\renewcommand{\ge}{\geqslant}
\renewcommand{\emptyset}{\varnothing}

\begin{document}

\title[]{Surface groups of diffeomorphisms of the interval}
\author{Ludovic Marquis}
\address{IRMAR, Universit\'e de Rennes 1}
\email{ludovic.marquis@univ-rennes1.fr}
\author{Juan Souto}
\address{IRMAR, Universit\'e de Rennes 1}
\email{juan.souto@univ-rennes1.fr}

\begin{abstract}
We prove that the group of diffeomorphisms of the interval $[0,1]$ contains surface groups whose action on $(0,1)$ has no global fix point and such that only countably many points of the interval $(0,1)$ have non-trivial stabiliser.
\end{abstract}

\maketitle

\section{Introduction}

The goal of this paper is to exhibit surface groups in the group $\Diff([0,1])$ of diffeomorphisms of the interval. To be precise, let
$$\Diff([0,1])=\{\phi\in\Diff_\infty(\BR)\vert\phi(t)=t\text{ for all }t\notin(0,1)\}$$
be the group of $C^{\infty}$-diffeomorphisms of the real line which are identity outside of the unit interval. Equivalently, $\Diff([0,1])$ is the group of smooth diffeomorphisms of the interval which are infinitely tangent to the identity at its endpoints. 

Let $\BF_2$ be the non-abelian free group in two generators. Since $\BF_2$ acts smoothly on the closed interval one can construct, considering a countable collection of disjoint subintervals of $[0,1]$, subgroups of $\Diff([0,1])$ isomorphic to the direct product of countably many free groups. As was noted by Baik, Kim and Koberda \cite{BKK}, this implies in turn that $\Diff([0,1])$ contains subgroups isomorphic to every fully residually free group, and thus in particular fundamental groups of closed surfaces. However, the dynamics of the so-obtained actions on the interval have many rather unpleasant properties. For instance, there are numerous global fixed points, there are no points with trivial stabiliser, and the action is not topologically transitive. Our goal is to prove that there also nicer surfaces groups in $\Diff([0,1])$:

\begin{sat}\label{main}
There are subgroups of $\Diff([0,1])$ isomorphic to the fundamental group of a closed surface of genus $2$ whose action on $(0,1)$ has no global fix point, is topologically transitive, and such that only countably many points of the interval $(0,1)$ have non-trivial stabiliser.
\end{sat}

Recall that a group action is topologically transitive if it has a dense orbit. Observe also that in some sense the Theorem \ref{main} is optimal from the point of view of the quantity of fixed points: H\"older's theorem \cite{Navas-book,Deroin-Navas-Rivas,Farb-Franks} implies that a group which acts by homeomorphism freely on an interval is abelian. Finally, lets us note that in Theorem \ref{main}, we consider surfaces of genus 2 just for the sake of concreteness. This is moreover no loss of generality because the fundamental group of a surface of genus $2$ contains the fundamental group of a surface of genus $g$ for all $g\ge 2$. 

Note now that Theorem \ref{main} admits an interpretation in terms of 3-manifolds. In fact, to every orientation preserving action $\pi_1(\Sigma)\actson\BR$ of a surface group on the real line one can associate a foliation on $\Sigma\times\BR$ as follows. Let $\tilde\Sigma$ be the universal cover of $\Sigma$, let $\pi_1(\Sigma)$ acts on it via deck-transformations, and consider the product action
$$\pi_1(\Sigma)\actson\tilde\Sigma\times\BR,\ \ \gamma\cdot(z,t)\mapsto(\gamma(z),\gamma(t)).$$
The quotient manifold is diffeomorphic to $\Sigma\times\BR$. Moreover, since this action preserves the foliation of $\tilde\Sigma\times\BR$ whose leaves are the planes $\tilde\Sigma\times\{t\}$, it follows that the quotient is naturally endowed with a foliation. Since standard generators of the groups constructed to prove Theorem \ref{main} can be chosen to be as close to the identity as one wishes, one gets for those groups that the obtained foliation is close to the trivial foliation with leaves $\Sigma\times\{t\}$. Moreover, since the action is trivial outside of the interval $[0,1]$, it follows that the obtained foliation is trivial outside of the compact set $\Sigma\times[0,1]$. Altogether we have:

\begin{kor}\label{kor-3mf foliation}
The trivial foliation of $\Sigma\times\BR$ can be smoothly perturbed within the compact set $C=\Sigma\times[0,1]$ so that $C$ is saturated and such that all but countably many leaves in $C$ are simply connected.\qed
\end{kor}

The proof of Theorem \ref{main} would be much easier (in fact, the result is basically folklore among people interested in these questions) if we were considering the group of homeomorphisms of the interval instead of the group of diffeomorphisms. In particular, the existence of such foliations as provided by Corollary \ref{kor-3mf foliation} is known in the topological category. In particular, it was already known, using a theorem of Calegari \cite{Calegari-smooth leaves}, that there were such foliations where each leaf is smooth. The difference between this statement and that of Corollary \ref{kor-3mf foliation} is that we are now also ensuring that the foliation is transversely smooth. These might look as a small difference but in general it is not. For example, if $S$ is a smooth negatively curved surface then the weak stable manifolds of the geodesic flow on $T^1S$ are always smooth, but the weak stable foliation is smooth if and only if the metric has constant curvature \cite{Ghys87}. Note that $T^1S$ is a 3-manifold and that the weak stable foliation has codimension 1. 

As we just mentioned, the main issue in the proof of Theorem \ref{main} is that we insist on the smoothness of the action. In fact, continuous and smooth actions on the interval are rather different. For example, a finitely generated group acts effectively by homeomorphisms on the interval if and only if it is left-orderable \cite{Deroin-Navas-Rivas} (and surface groups have a Cantor set worth of orders \cite{ABR}), but this condition is far from ensuring the existence of smooth actions. For example, Thurston's stability theorem \cite{Thurston} asserts that the group of $C^1$-diffeomorphisms on the interval is locally indicable, meaning that any (non-trivial) finitely generated subgroup surjects on $\BZ$. In terms of orders, this means that every group acting effectively by $C^1$-diffeomorphisms on the interval admits what is called a $C$-order (see \cite{Deroin-Navas-Rivas} for the relation between orders on groups and one-dimensional dynamics). However, there are many examples of locally indicable groups which do not act smoothly on the interval. For example, Navas \cite{Navas} proved that there are semi-direct products $\BZ^2\rtimes\BF_2$ which do not admit effective $C^1$-actions on $[0,1]$. Other examples of this phenomenon are due to Calegari \cite{Calegari} and Bonatti-Monteverde-Navas-Rivas \cite{BMNR}.

The reader could be by now thinking that those comments are all nice and well, but that surface groups are possibly the most flexible groups after free groups, and that there are many instances in which it is known that if a group contains a free group then it also contains a surface group. We agree. For example, using that surface groups are limits of free groups, it was proved in \cite{BGSS} that a locally compact group $G$ which contains a non-discrete free group, also contains a surface group. There are several ways to present the argument from \cite{BGSS} but the simplest form of the argument, and the one which is most prone to generalisation, goes as follows. First consider the group $\pi_1(\Sigma)$ as an amalgamated product
$$\pi_1(\Sigma)=\BF_2*_\BZ\BF_2$$
where the amalgamation is given by identifying the commutators $[a,b]=[a',b']$ for the free bases $\{a,b\}$ and $\{a',b'\}$ of the first and second copies of $\BF_2$ respectively. Now, take $a,b\in G$ which generate a free group and such that there is a 1-parameter subgroup $(g_t)\subset G$ with $[a,b]=g_1$. Now, definitively if $G$ is a Lie group but also in all cases that come to mind to the authors, one has that almost all representations
$$\rho_t:\pi_1(\Sigma)\to G,\ \ \rho_t(a)=a,\ \rho_t(b)=b,\ \rho_t(a')=g_ta g_t^{-1},\ \rho_t(b')=g_tbg_t^{-1}$$
are faithful. This is a very flexible argument. But we do not know how to make it work if $G=\Diff_\infty([0,1])$ because we do not know how to ensure that the commutator $[a,b]$ is part of a flow. In fact, centralisers of generic diffeomorphisms are cyclic groups \cite{Bonatti-Crovisier-Wilkinson}.

In fact, the argument used to prove Theorem \ref{main} is of a very different nature. Basically, if $\Sigma$ is our closed surface, we will obtain the desired homomorphism
$$\rho:\pi_1(\Sigma)\to\Diff_\infty([0,1])$$
as a limit $\rho=\lim\rho_n$. Here, the approximating homomorphisms $\rho_n$ will be constructed inductively in such a way that at each time the corresponding kernel is contained in a certain subgroup $\Gamma_n$ of $\pi_1(\Sigma)$ satisfying
$$\pi_1(\Sigma)=\Gamma_0\triangleright\Gamma_1\triangleright\Gamma_2\triangleright\ldots\text{ with }\cap_{n=1}^\infty\Gamma_n=\{ \Id \}\text{ and }\Gamma_{n-1}/\Gamma_n=\BZ\text{ for all }n.$$
We will obtain $\rho_n$ as the holonomy of a certain perturbation of the foliation of $\Sigma\times\BR$ associated to $\rho_{n-1}$. The basic two ingredients of the construction of this perturbation are
\begin{itemize}
\item that the foliation associated to $\rho_{n-1}$ contains a trivially foliated product $\BH^2/\Gamma_{n-1}\times I_{n-1}$ for some interval $I_{n-1}\subset[0,1]$, and 
\item that the cohomology class in $H^1(\BH^2/\Gamma_{n-1};\BZ)$ given by the isomorphism $\Gamma_{n-1}/\Gamma_n=\BZ$ is compactly supported.
\end{itemize}
It is this last point the one which makes us wonder if groups such as the fundamental group of a hyperbolic 3-manifold $M$ fibering over the circle and with say $H_1(M;\BZ)=\BZ$ can be subgroups of $\Diff_\infty([0,1])$.
\medskip

The paper is organised as follows. In section \ref{sec-algebra} we construct the filtration of $\pi_1(\Sigma)$ mentioned above. Then, in section \ref{sec-proof of main} we prove Theorem \ref{main} assuming Proposition \ref{prop-approx}, the key technical step in this paper. In section \ref{sec-review} we recall the dictionary between foliations and their holonomies. This dictionary is the key to prove Proposition \ref{prop-approx} in section \ref{sec-meat}.
\medskip

\noindent{\bf Acknowledgements.} We would like to thank Bill Breslin and specially Sang-hyun Kim for very interesting conversations on this topic. We also thank the referee for his or her useful report.

\section{A usefull unscrewing of surface groups}\label{sec-algebra}

Let from now on $\Sigma$ be a closed hyperbolic surface of genus $2$ and identify its universal cover with the hyperbolic plane $\BH^2$. Choose a base point $*\in\BH^2$ and denote again by $*$ its image under the cover $\BH^2\to\Sigma$. In the same way, if $\Gamma\subset\pi_1(\Sigma,*)$ is any subgroup then we denote by $*$ the projection of the base point of $\BH^2$ to $\BH^2/\Gamma$. Here we let $\pi_1(\Sigma,*)$ act on $\BH^2$ via deck-transformations. Finally, even if most of the time we do not make this explicit, curves in $\Sigma$ and its covers will be assumed to be oriented. In any case, the chosen orientation will basically play no role in the arguments.

Anyways, in this section we construct a certain decreasing filtration of the surface group $\pi_1(\Sigma,*)$. More precisely we prove:

\begin{prop}\label{prop-filtration}
There is a decreasing sequence of nested subgroups 
$$\pi_1(\Sigma,*)=\Gamma_0\supset\Gamma_1\supset\Gamma_2\supset\dots$$
of the fundamental group of $\Sigma$ satisfying 
$$\cap_i\Gamma_i=\{\Id\}$$ 
and such that for all $i$ we have
$$\Gamma_{i+1}=\{\gamma\in\Gamma_i\vert\langle\gamma,c_i\rangle=0\}$$
where $c_i\subset\BH^2/\Gamma_i$ is an (oriented) simple closed curve and where $\langle\cdot,\cdot\rangle$ is the algebraic intersection number on $\BH^2/\Gamma_i$.
\end{prop}

Before launching the proof we need a definition and a simple fact. We will say that a hyperbolic surface $X$ has {\em genus all over the place} if there is some $R>0$ such that for all $x\in X$ there are  pair of simple closed curves $\alpha_x$ and $\beta_x$ contained in the ball of radius $R$ centred at $x$ and intersecting transversally exactly once. In symbols this means that 
\begin{equation}\label{eq-genus all over the place}
\langle\alpha_x,\beta_x\rangle=1\text{ and }\alpha_x,\beta_x\subset B^X(x,R).
\end{equation}
To prove Proposition \ref{prop-filtration} we will use the following fact:

\begin{lem}\label{lem-genus all over}
Let $X=\BH^2/\pi_1(X)$ be a connected hyperbolic surface, $c\subset X$ a simple closed curve, and $\Gamma=\{\gamma\in\pi_1(X)\vert\langle c,\gamma\rangle=0\}$. If $X$ has genus all over the place, then so does $Y=\BH^2/\Gamma$ as well.
\end{lem}

\begin{proof}
First note that if $c$ is separating, then there is nothing to be proved because $X=Y$ in this case. We assume from now on that $c$ is non-separating.

Suppose first that $X$ is closed of genus $g$. The cover $Y\to X$ is cyclic and admits a surface of genus $g-1>0$ as a fundamental domain for this $\BZ$ action. So, $Y$ has genus all over the place with $R$ equal twice the diameter of the fundamental domain. We suppose from now on that $X$ is not closed.

The set $K=\{x\in X\vert d_X(x,c)\le R+1\}$ is compact in $X$. Let $\tilde K=\pi^{-1}(K)\subset Y$ be its preimage under the cover $\pi:Y\to X$ and let $\hat K$ be the union of $\tilde K$ and of all bounded connected components of $Y\setminus\tilde K$. Since the cover $\pi:Y\to X$ is normal (with deck-transformation group $\BZ$), it follows that $\tilde K$ has a compact fundamental domain under the action $\BZ\actson\tilde K\subset Y$. This implies then in turn that $\hat K$ also has a compact fundamental domain. It follows that there is some $R'$ such that for all $y\in Y$ there is $y'\in B^Y(y,R')$ such that $d_X(\pi(y'),c)\ge R+1$. Now, for any such $y'$ we have that the ball $B^X(\pi(y'),R)$ centred at its projection $\pi(y')$ and with radius $R$ is disjoint of $c$ and hence lifts isometrically to $Y$. In symbols this means that
$$B^Y(y',R) \underset{isometric}{\simeq} B^X(\pi(y'),R).$$
It follows that $B^Y(y',R)$ contains a pair of curves which meet transversally and exactly once. Since $B^Y(y',R)\subset B^Y(y,R'+R)$ we have thus proved that $Y$ has genus all over the place, as we had claimed.
\end{proof}

Armed with Lemma \ref{lem-genus all over} we are ready to prove Proposition \ref{prop-filtration}:

\begin{proof}[Proof of Proposition \ref{prop-filtration}]
We start by ordering the nontrivial elements of $\pi_1(\Sigma)=\Gamma_0$ by length, meaning that we choose a total order $\le$ satisfying $\ell_\Sigma(\gamma)\le\ell_\Sigma(\eta)$ for all $\gamma\le\eta$. Here $\ell_\Sigma$ is the hyperbolic length of the shortest loop of $\Sigma$ based at $*$ in the homotopy class relative to the base point $*$. Armed with this order we start the construction of the filtration $(\Gamma_i)$ and of the respective sequence of curves.

We will work inductively, starting with $\Gamma_0=\pi_1(\Sigma,*)$. Suppose that we have given $\Gamma_n$ and let $\eta$ be a smallest non-trivial element of $\Gamma_n$ with respect to the order $\le$. We will construct $(\Gamma_i)$ in such a way that $\eta\notin\Gamma_{n+2}$. Note that this suffices to show that $\cap\Gamma_i=\Id$.

Starting with the construction of $\Gamma_{n+1}$, note that the choice of $\eta$ as a smallest non-trivial element in $\Gamma_n$, ensures that $\eta$ is a shortest homotopically non-trivial loop in $X=\BH^2/\Gamma_n$ based at the base point $*$. In particular, $\eta$ is a simple loop, that is, without self-intersections. If $\eta$ is non-separating then let $c_n$ be a simple closed curve in $X$ which meets $\eta$ exactly once. This means that 
$$
\eta\notin\Gamma_{n+1}\stackrel{\text{def}}=\{\gamma\in\Gamma_n\vert\langle\gamma,c_n\rangle=0\}.
$$
We can thus suppose that $\eta$ separates $X$ into two components $X_1,X_2$. 

Note that if $X_1$ is compact, then it has positive genus because it has connected boundary $\eta$ and is a compact subsurface of a hyperbolic surface - in particular, $X_1$ contains a pair of simple closed curves $\alpha_1,\beta_1$ intersecting transversally exactly once. We claim that the same is true also if $X_1$ is not compact. To see that this is the case note first that, using Lemma \ref{lem-genus all over} inductively, we get that $X$ has genus all over the place. Let $R>0$ be as in the definition of having genus all over the place, choose $x_1\in X_1$ with $d_X(x_1,\eta)\ge R+1$ and pick one of the pair of simple closed curves $\alpha_1,\beta_1$ in $B^X(x_1,R)$ guaranteed by the definition of having genus all over the place. Reversing the roles of $X_1$ and $X_2$ we have then proved:
\medskip

\noindent{\bf Fact.} {\em Both components $X_1$ and $X_2$ of $X\setminus\eta$ contain a pair $\{ \alpha_i,\beta_i\}$ of simple closed curves which meet transversally and exactly once.\qed}
\medskip

\begin{figure}
\centering
\includegraphics[scale=0.5]{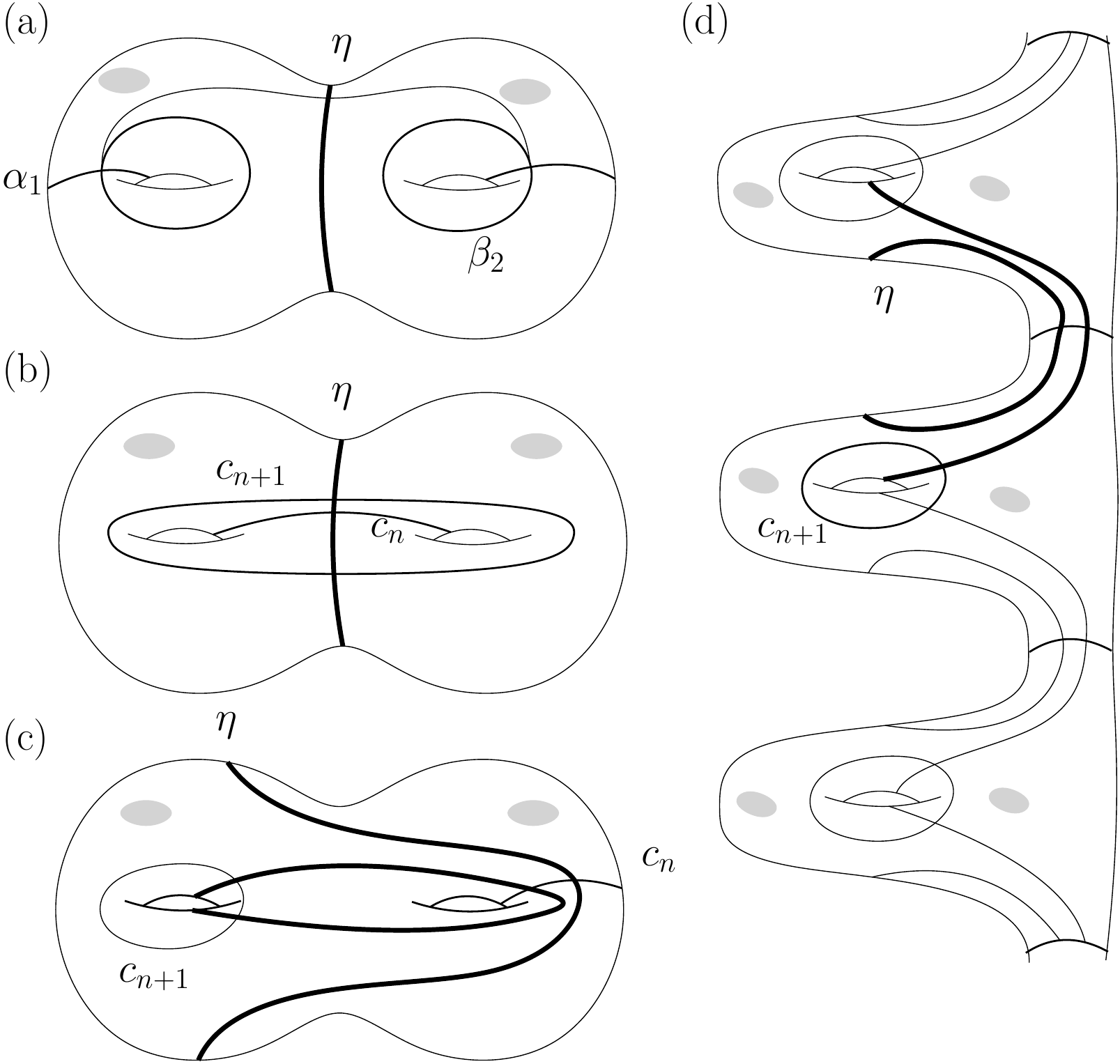}
\caption
{(a) We see the surface $Z\subset X$ (the grey islands denote the connected components of $X\setminus Z$), the curves $\eta,\alpha_1,\beta_1,\alpha_2,\beta_2$ and the arc $J$. (b) We see again $Z$ and $\eta$, and now also the curves $c_n$ and $c_{n+1}$. (c) We see again the same as in (b), but from a different point of view, that is after applying a mapping class. Finally, in (d) we see the cover $Y=\BH^2/\Gamma_n$ with lifts of $\eta$, $c_n$ and $c_{n+1}$, and with the homeomorphically lifted grey islands. The important point to note is that curves $\eta$ and $c_{n+1}$ in $Y$ intersect exactly once.}\label{bigfigure}
\end{figure}

Let $\alpha_1,\beta_1\subset X_1$ and $\alpha_2, \beta_2\subset X_2$ be the four curves provided by the fact, oriented in such a way that $\langle\alpha_i,\beta_i\rangle=1$ for $i=1,2$. Let also $J$ be an embedded arc in $X$ joining the points $\alpha_1\cap\beta_1$ and $\alpha_2\cap\beta_2$, whose interior is disjoint of the curves $\alpha_1,\beta_1,\alpha_2$ and $\beta_2$, and which meets $\eta$ exactly once (compare with (a) in figure \ref{bigfigure}). Let $c_n$ be the simple closed curve obtained from $\alpha_1$ and $\alpha_2$ by surgery along $J$. In other words, $c_n$ is the component of the boundary of a regular neighborhood of $\alpha_1\cup J\cup\alpha_2$ which is not isotopic to one of the $\alpha_i$. Similarly, let $c_{n+1}$ be the curve obtained from $\beta_1$ and $\beta_2$ via surgery along $J$ (compare with (b) in figure \ref{bigfigure}). Note that a regular neighborhood of the union of the curves $\alpha_1,\beta_1,\alpha_2,\beta_2,\eta$ and the arc $J$ is a subsurface $Z$ of $X$ homeomorphic to a surface of genus $2$ with $2$ boundary components.

\begin{bem}
Although we will not need it below, note that up to reversing the orientation of $\alpha_1$ and/or $\alpha_2$, and thus of $\beta_1$ and/or $\beta_2$, we might assume that $c_n$ is homologous to $\alpha_1+\alpha_2$, and that $c_{n+1}$ is homologous to $\beta_1-\beta_2$.
\end{bem}

Continuing with the proof, take now $\Gamma_{n+1}=\{\gamma\in\Gamma_n\vert\langle\gamma,c_n\rangle=0\}$ and note that, unfortunately, $\eta\in\Gamma_{n+1}$. This means that the loop $\eta$ in $X=\BH^2/\Gamma_n$ lifts to a loop in the cover $Y=\BH^2/\Gamma_{n+1}$ (see (d) in figure \ref{bigfigure} for a pictorial representation of $Y$). However, the so-obtained loop, which we again denote by $\eta$, does not separate $Y$. Indeed, it meets some lift of $c_{n+1}$ to $Y$ in a single point. Denoting again this lift by $c_{n+1}$ we have that
$$\eta\notin\Gamma_{n+2}\stackrel{\text{def}}=\{\gamma\in\Gamma_{n+1}=\pi_1(Y,*)\vert\langle\gamma,c_{n+1}\rangle=0\},$$
as we needed to prove.
% 
%It only remains to show the claim. First get back to the picture A which represents the curves $\eta$, $c_n=\alpha_1+\alpha_2$ and $c_{n+1}=\beta_1-\beta_2$ on $X$ in such a way that $\eta$ is the simplest separating curve one can imagine. Now, apply one \textcolor{red}{can find}\footnote{I will be happy to put a reference to this point, is the primer on MCG of Farb and Margalit, the good one ?} a mapping class to be in the situation described by picture B where this time it is $c_n$ which is the simplest non-separating curve one can imagine. The existence of a mapping class exchanging the two pictures is guaranted by the fact that the algebraic intersection numbers of the three curves are the same in both pictures.
%
%Finally, the cover $Y$ of $X$ is given by the picture C, it should be clear now that $c_{n+1}$ fulfill the claim. This concludes the induction step, and thus the proof of Proposition \ref{prop-filtration}.
\end{proof}

\section{Proof of the main theorem assuming the key step}\label{sec-proof of main}

The groups whose existence is claimed in Theorem \ref{main} will be constructed by a limiting process. The key step is the construction of a sequence of homomorphisms $\pi_1(\Sigma,*)\to\Diff([0,1])$ as follows. 

\begin{prop}\label{prop-approx}
Let $(\Gamma_n)$ be the sequence of groups provided by Proposition \ref{prop-filtration} and fix a finite symmetric generating set $\CS$ of $\pi_1(\Sigma,*)$.

For all $\epsilon>0$ there is a sequence of pairs $(\rho_n,I_n)$ where 
$$\rho_n:\pi_1(\Sigma,*)\to\Diff([0,1])$$
is a homomorphism and $I_n$ is an open subinterval of $[0,1]$, starting with the trivial homomorphism $\rho_0$ and with the interval $I_0=(0,1)$, and such that for all $n\ge 1 $ the following conditions are satisfied:
\begin{enumerate}
\item The closure of $I_n$ is contained in $I_{n-1}$, that is $\overline{I_n}\subset I_{n-1}$.

\item $\rho_n$ is $C^n$-close to $\rho_{n-1}$, meaning that 
$$\Vert\rho_n(\gamma)-\rho_{n-1}(\gamma)\Vert_{C^n}\le 10^{-n}\cdot\epsilon$$ 
for every $\gamma\in\CS$. Here, $\Vert f\Vert_{C^n}=\max_{i=0}^n\Vert f^{(i)}\Vert_\infty$ is the standard $C^n$-norm.

\item $\Stab_{\rho_n}(x)=\Gamma_n$ for all $x\in I_n$, where
$$\Stab_{\rho_n}(x)=\{\gamma\in\pi_1(\Sigma)\vert \rho_n(\gamma)(x)=x\}$$ 
is the stabiliser of $x$ with respect to the action of $\pi_1(\Sigma,*)$ on $[0,1]$ induced by $\rho_n$.

\item For all $\gamma\notin\Gamma_n$ we have $\rho_n(\gamma)(I_n)\cap I_n=\emptyset$. In particular, 
$$\D\overline I_n\subset K_n\stackrel{\text{def}}=  [ 0,1 ]\setminus\bigcup_{\gamma\in\pi_1(\Sigma,*)}\rho_n(\gamma)(I_n).$$

\item Finally, $\rho_n(\gamma)(x)=\rho_{n-1}(\gamma)(x)$ for all $\gamma\in\pi_1(\Sigma,*)$ and all $x\in K_{n-1}$. In particular, $K_{n-1}\subset K_n$.
\end{enumerate}
\end{prop}

We will prove Proposition \ref{prop-approx} in section \ref{sec-meat}. Assuming it for now, we conclude the proof of Theorem \ref{main}.

\begin{proof}[Proof of the injectivity part of Theorem \ref{main}]
First note that (2) in Proposition \ref{prop-approx} implies that the sequence $(\rho_k(\gamma))_k$ is Cauchy-sequence with respect to the $C^n$-norm for all $n$. It follows that the limit $\lim_{k\to\infty}\rho_k(\gamma)$ exists and that it is smooth. In other words we get that the homomorphism $\rho_k$ converge when $k\to\infty$ to a homomorphism
$$\rho:\pi_1(\Sigma,*)\to\Diff_\infty([0,1]),\ \ \rho(\gamma)\stackrel{\text{def}}=\lim_{k\to\infty}\rho_k(\gamma).$$

To prove Theorem \ref{main} we only show for the moment that $\rho$ is injective. To see that this is the case fix $n$ and let $x\in\D\overline{I_n}$. By (4) and (5) in the proposition we have that
$$x\in\D\overline I_n\subset K_n\subset K_{n+1}\subset K_{n+2}\subset\ldots$$
It thus follows from (5) that 
$$\rho_n(\gamma)(x)=\rho_{n+1}(\gamma)(x)=\rho_{n+2}(\gamma)(x)=\ldots.$$
for all $\gamma\in\pi_1(\Sigma)$. This implies that we also have
$$\rho(\gamma)(x)=\rho_n(\gamma)(x)$$
for all $\gamma$. Now, because of (4) we have that
$$\Stab_{\rho}(x)=\Stab_{\rho_n}(x)\subset\Gamma_n$$
Since $\cap\Gamma_n=\Id$, the claim follows.
\end{proof}

The fact that Proposition \ref{prop-approx} is true for every $\epsilon$ implies that we can choose the homomorphism $\rho$ as close (on the generators) as we want to the trivial homomorphism, which in return implies that the foliation given by such a homomorphism $\rho$ is a smooth pertubation of the trivial foliation, so give us Corollary \ref{kor-3mf foliation}.

\section{Review about foliation and holonomy}\label{sec-review}

It remains to prove Proposition \ref{prop-approx}. As we mentioned above, we will do that in section \ref{sec-meat} but first recall briefly the dictionary between actions on the real line and co-dimension one foliations. 

Suppose that we have a locally trivial line bundle $\pi:M\to\Sigma$, by which we mean that $M$ is a smooth fiber bundle with fibers diffeomorphic to $\BR$. If the total space $M$ is endowed with a codimension one smooth foliation $\CF$ transversal to the fibers of $\pi$, then the distribution $\{x\mapsto T_x\CF\}$ of planes tangent to $\CF$ is a smooth flat connection on $M$. If moreover, every leaf of $\CF$ is contained in a compact set of $M$ then the connection is complete, meaning that parallel transport exists for all times. Note for example that this condition is satisfied if there is a $\CF$-saturated compact $C\subset M$ such that the restriction of $\pi$ to any leaf of $\CF$ not contained in $C$ is a diffeomorphism onto $\Sigma$. Suppose that such a compact set $C$ exists. It follows thus that the holonomy representation
$$\rho_\CF:\pi_1(\Sigma,*)\to\Diff(\pi^{-1}(*))$$
is well-defined and that its image $\rho_\CF(\pi_1(\Sigma,*))$ fixes every point outside of $\pi^{-1}(*)\cap C$. In other words, if we have an identification of $\pi^{-1}(*)$ with $\BR$ in such a way that $\pi^{-1}(*)\cap C\subset[0,1]$, then the holonomy representation takes values in $\Diff([0,1])$.

We also note that if we fix $\gamma\in\pi_1(\Sigma,*)$, and if we are given a second foliation $\CF'$ with the same properties as above, and such that the distributions $\{x\mapsto T_x\CF\}$ and $\{x\mapsto T_x\CF'\}$ are $C^k$-close as sections of $\Lambda^2TM\to M$, then the images $\rho_{\CF'}(\gamma)$ and $\rho_\CF(\gamma)$ of $\gamma$ under the new and the old holonomy representations are also $C^k$-close to each other.

Turning now the tables, suppose that we are given a representation 
$$\sigma:\pi_1(\Sigma,*)\to\Diff([0,1])$$
and recall that we are identifying $\BH^2$ with the universal cover of $\Sigma$. Endowing thus $\BH^2$ with the deck-transformation action of the fundamental group $\pi_1(\Sigma,*)$ we consider the product action
\begin{equation}\label{eq-action}
\pi_1(\Sigma,*)\actson_\sigma \BH^2\times\BR,\ \ \gamma\cdot(z,t)\mapsto(\gamma(z),\sigma(\gamma)(t))
\end{equation}
whose second factor is given by $\sigma$. The action \eqref{eq-action} is discrete, meaning that that the quotient space $M_\sigma$ is a manifold. Note now that projection $\BH^2\times\BR\to\BH^2$ induces a map 
$$\pi:M_\sigma\to\BH^2/\pi_1(\Sigma,*)=\Sigma.$$
In fact, $\pi:M_\sigma\to\Sigma$ is a locally trivial smooth line bundle over $\Sigma$. Moreover, the total space $M_\sigma$ is endowed with a codimension one foliation $\CF_\sigma$. To see that this is the case note that the action \eqref{eq-action} preserves the foliation of $\BH^2\times\BR$ whose leaves are of the form $\BH^2\times\{t\}$. This foliation being preserved by \eqref{eq-action}, it descends to a foliation $\CF_\sigma$ of $M_\sigma$. 

These two processes we just described are inverse to each other. In fact, noting that the foliation $\CF_\sigma$ is transversal to the fibers of $\pi$, note also that the image of $\BH^2\times[0,1]$ under the quotient map $\BH^2\times\BR\to M_\sigma$ is a saturated compact set $C$ such that the restriction of $\pi$ to any fiber of $\CF$ not contained in $C$ is a diffeomorphism onto $\Sigma$. Note also that we have obvious identifications $\pi^{-1}(*)=*\times\BR=\BR$ and that under those identifications the set $C\cap\pi^{-1}(*)$ goes to the closed interval $[0,1]$. It follows that the holonomy representation $\rho_{\CF_\sigma}$ of $\CF_\sigma$ takes values in $\Diff([0,1])$. In fact $\rho_{\CF_\sigma}=\sigma$.

Conversely, if we start with a foliation $\CF$ of the total space of a line bundle $\pi:M\to\Sigma$ with the properties above, we consider the associated holonomy representation $\rho_\CF$, and then construct the associated line bundle $\pi:M_{\rho_\CF}\to\Sigma$ and foliation $\CF_{\rho_\CF}$ then we have a bundle isomorphism 
$$\xymatrix{M\ar[dr]\ar[rr] & & M_{\rho_\CF}\ar[dl] \\ & \Sigma & }$$
mapping $\CF$ to $\CF_{\rho_\CF}$.
\medskip

\begin{bem}
Recall that, if $z\in \pi^{-1}(*)$ then the restriction of $\pi$ to the leaf $\CF_z$ of $\CF$ containing $z$ is a covering of $ \Sigma$, whose fundamental group is the subgroup of $\pi_1(\Sigma,*)$ which is the stabilizer of $z$.
\end{bem}

\noindent{\bf Restricting to saturated sets.} Before concluding this digression, suppose that we have a line bundle $\pi:M\to\Sigma$, a codimension  one foliation $\CF$ on $M$ transversal to the fibers of $\pi$, and a compact set $C\subset M$ with the property that the restriction of $\pi$ to each leaf which is not contained in $C$ is a diffeomorphism. Letting as always $*\ni\Sigma$ be the base point, identify as above $\pi^{-1}(*)=\BR$ in such a way that $\pi^{-1}(*)\cap C=[0,1]$. Suppose now that $U$ is an open saturated subset of $C$ and suppose that there is a connected component $J$ of $\pi^{-1}(*)\cap U$ with the property that every leaf of $\CF$ contained in $U$ meets $J$. 

\begin{figure}[H]
\centering
\includegraphics[scale=0.5]{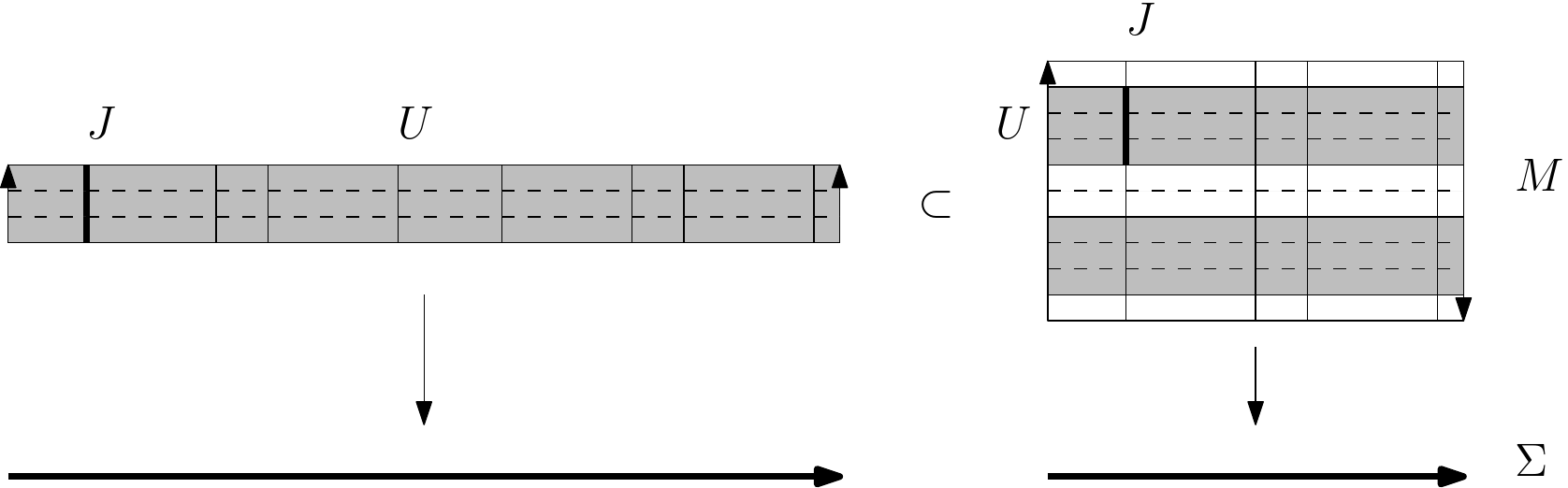}
\bigskip
\bigskip
\caption{Consider the Moebius band $M$ as the total space of a line bundle with base $\Sigma=\BS^1$, foliate $M$ by circles, let $U\subset M$ be a saturated annulus and restrict the foliation to $U$ - possibly a too simple example, but the only one we can possibly draw.}\label{moebius}
\end{figure}

\begin{bem}
Such a $J$ exists if $U$ is connected and it can be chosen to be any connected component of $\pi^{-1}(*)\cap U$. 
\end{bem}

Note that since $J$ is an interval, the relative homotopy group $\pi_1(U,J)$ is isomorphic to the fundamental group of $U$. Note also that, applying the long homotopy sequence to the bundle $U\to\Sigma$, we get that the restriction $\pi\vert_U$ of the projection $\pi$ to the open set $U$ induces an injective homomorphism
$$(\pi\vert_U)_*:\pi_1(U,J)\to\pi_1(\Sigma,*).$$
The restriction of $\pi$ to $U$ lifts to an interval bundle
$$\pi\vert_U:U\to\BH^2/(\pi\vert_U)_*(\pi_1(U,J))$$
and the holonomy of the induced foliation $\CF\vert U$ is related to the holonomy of the foliation $\CF$ as follows: 
$$\rho_\CF((\pi\vert_U)_*(\gamma))\vert_J=\rho_{\CF\vert U}(\gamma)\text{ for all }\gamma\in\pi_1(U,J).$$
In other words, as long as we restrict to the group $\pi_1(U,J)$ and we restrict the holonomies to $J$, then the holonomy of $M\to\Sigma$ and the holonomy of $U\to\Sigma$ agree.

\section{Proof of the key step}\label{sec-meat}

We are now ready to prove Proposition \ref{prop-approx}. We will construct the desired pairs ($\rho_n,I_n)$ by induction. Since we have already that $I_0=(0,1)$ and that $\rho_0$ is the trivial homomorphism, we can assume by induction, that the pair $(\rho_{n-1},I_{n-1})$ has been already constructed. To construct the next pair we will consider the foliation $\CF_{\rho_{n-1}}$ associated to $\rho_{n-1}$, perturb it to a new foliation, and take the associated holonomy.

Starting thus with the representation $\rho_{n-1}$ we consider as above the product action $\pi_1(\Sigma)\actson_{\rho_{n-1}}\BH^2\times\BR$ as in \eqref{eq-action} and let
$$M_{n-1}=M_{\rho_{n-1}}=\BH^2\times\BR/_{\{(x,t)\sim(\gamma(x),\rho_{n-1}(\gamma)(t))\vert \gamma\in\pi_1(\Sigma,*)\}}$$
be the total space of the associated line bundle $\pi:M_{n-1}\to\Sigma$. As before, we have a canonical identification $\pi^{-1}(*)=\BR$. Denote by $C\subset M_{n-1}$ the projection of $\BH^2\times[0,1]$ and note that $C$ is compact. Note also that $C\cap\pi^{-1}(*)=[0,1]$. Finally, let $\CF_{n-1}=\CF_{\rho_{n-1}}$ be the 2-dimensional foliation associated to $\rho_{n-1}$. 

Consider now the obvious embedding
\begin{equation}\label{eq-blabla}
\BH^2\times I_{n-1}\to\BH^2\times\BR,\ (z,t)\mapsto(z,t).
\end{equation}
By induction we know that $(\rho_{n-1},I_{n-1})$ satisfies (3) in Proposition \ref{prop-approx}, meaning for starters that $\Gamma_{n-1}$ acts trivially on $I_{n-1}$. It follows that the embedding \eqref{eq-blabla} descends to a well-defined map
\begin{equation}\label{eq-blabla2}
\phi:\BH^2/\Gamma_{n-1}\times I_{n-1}\to M_{n-1}.
\end{equation}
Since $(\rho_{n-1},I_{n-1})$ also satisfies (4) in Proposition \ref{prop-approx}, we also know that every element of $\pi_1(\Sigma,*)\setminus\Gamma_{n-1}$ moves $I_{n-1}$ off itself. It follows that \eqref{eq-blabla2} is an embedding. Since the image of \eqref{eq-blabla} is saturated under the foliation of $\BH^2\times\BR$ by planes $\BH^2\times\{t\}$ we derive that \eqref{eq-blabla2} maps leaves of the trivial foliation to leaves of $\CF_{n-1}$ and that its image $\phi(\BH^2/\Gamma_{n-1}\times I_{n-1})\subset M_{n-1}$ is saturated. 

Anyways, we summarize the situation at hand in the following lemma:

\begin{lem}\label{lem-product}
There is an embedding
$$\phi:\BH^2/\Gamma_{n-1}\times I_{n-1}\to C\subset M_{n-1}$$
with the following properties:
\begin{itemize}
\item[(a)] $\phi(\BH^2/\Gamma_{n-1}\times\{t\})$ is a leaf of $\CF_{n-1}$ for all $t\in I_{n-1}$.
\item[(b)] $\phi(\{z\}\times I_{n-1})$ is contained in a fiber of the projection $\pi: M_{n-1}\to\Sigma$ for all $z\in\BH^2/\Gamma_{n-1}$.
\item[(c)] $\pi(\phi(\{*\}\times I_{n-1}))=*$ and, after our earlier identification $\pi^{-1}(*)=\BR$, we have that in fact $\phi(\{*\}\times I_{n-1})=I_{n-1}$.
\item[(d)] Moreover, $\phi(\BH^2/\Gamma_{n-1}\times I_{n-1})\cap\pi^{-1}(*)=\cup_{\gamma\in\pi_1(\Sigma,*)}\rho_{n-1}(\gamma)(I_{n-1})$.
\qed
\end{itemize}
\end{lem}

The new homomorphism $\rho_n:\pi_1(\Sigma,*)\to\Diff([0,1])$ will be obtained as the holonomy of a foliation $\CF$ on $M_{n-1}$, transversal to the fibers of the projection $\pi:M_{n-1}\to\Sigma$, and obtained by perturbing $\CF_{n-1}$ within the image of the embedding $\phi$ from Lemma \ref{lem-product}. We note before going any further that it follows from Lemma \ref{lem-product} (d) that the holonomy $\rho_\CF$ of any such foliation satisfies condition (5) in Proposition \ref{prop-approx}. Using the induction hypothesis, the holonomy of any such $\CF$ also satisfies the (4) of proposition \ref{prop-approx} at the stage $n-1$, meaning:
\begin{itemize}
\item[(4')] For any $\gamma\notin\Gamma_{n-1}$ we have $\rho_{\CF}(\gamma)(I_{n-1})\cap I_{n-1}=\emptyset$.
\end{itemize}
Well, we start with the construction of a concrete perturbation $\CF$ of $\CF_{n-1}$. As we just mentioned, everything is going to happen within the image of $\phi$. 

We cannot recall too often that the restriction of $\CF_{n-1}$ to $\phi(\BH^2/\Gamma_{n-1}\times I_{n-1})\simeq \BH^2/\Gamma_{n-1}\times I_{n-1}$ is nothing but the foliation whose leaves are the copies $\BH^2/\Gamma_{n-1}\times\{t\}$ of the cover $\BH^2/\Gamma_{n-1}$ of $\Sigma$ associated to $\Gamma_{n-1}$. Recall also that by the very construction of the group $\Gamma_n$ in Proposition \ref{prop-filtration}, there is a simple closed curve $c\subset\BH^2/\Gamma_{n-1}$ such that 
$$\Gamma_n=\{\gamma\in\pi_1(\Sigma,*)\vert \langle\gamma,c\rangle=0\}.$$
We might assume that $c$ does not meet the base point $*$. Consider now an embedding
$$\psi:\BS^1\times[-1,1]\to\BH^2/\Gamma_{n-1}\setminus\{*\}$$
with $\psi(\BS^1\times\{0\})=c$ and the corresponding embedding
$$(\psi\times\Id):(\BS^1\times[-1,1])\times I_{n-1}\to\BH^2/\Gamma_{n-1}\times I_{n-1}$$
We are going to perturb the trivial foliation of $\BH^2/\Gamma_{n-1}\times I_{n-1}$ inside the image of $\psi\times\Id$. To do so we replace the trivial foliation of $\BS^1\times[-1,1]\times I_{n-1}$ by cylinders $\BS^1\times[-1,1]\times\{t\}$ by a new foliation by cylinders as suggested by Figure \ref{fig1}.

\begin{figure}[h]
\includegraphics[width=8cm, height=4.5cm]{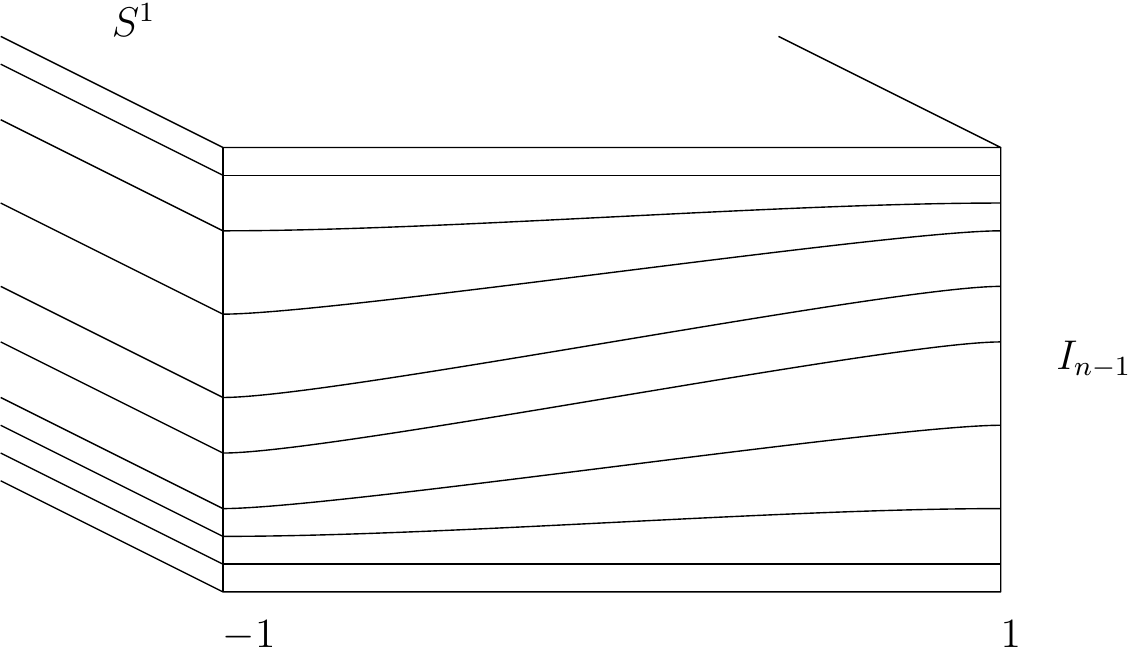}
\caption{}
\label{fig1}
\end{figure}

In more precise, but definitively more obscure terms, choose a diffeomorphism
\begin{equation}\label{eq-diffeo}
f\in\Diff(\overline{I_{n-1}})
\end{equation}
with $f(y)>y$ for all $y\in I_{n-1}$ and a smooth monotone function $g:[-1,1]\to[0,1]$ with $g(x)=0$ for $x$ near $-1$ and $g(x)=1$ for $x$ near $1$. Now consider the foliation $\CG$ of $\BS^1\times[-1,1]\times I_{n-1}$ whose leaf through $(\theta,y,-1)$ is the graph of the function
$$\BS^1\times[-1,1]\to I_{n-1},\ \ (\theta,x)\mapsto (1-g(x))y+g(x)f(y).$$
Now, the foliation $\CG$ agrees near $\BS^1\times\{-1,1\}\times I_{n-1}$ with the trivial foliation with leaves $\BS^1\times[-1,1]\times\{y\}$. In particular, we can extend the foliation $(\psi\times\Id)(\CG)$ on $(\psi\times\Id)(\BS^1\times[-1,1]\times I_{n-1})$ to a foliation $\CH$ on $\BH^2/\Gamma_{n-1}\times I_{n-1}$ by declaring that $\CH$ agrees with the trivial foliation outside of $(\psi\times\Id)(\CG)$.

We now let $\CF_n$ be the foliation of $M_{n-1}$ which agrees with $\CF_{n-1}$ on $M_{n-1}\setminus\phi(\BH^2/\Gamma_{n-1}\times I_{n-1})$ and with $\phi(\CH)$ on $\phi(\BH^2/\Gamma_{n-1}\times I_{n-1})$. Here $\phi$ is, as all along, the embedding \eqref{eq-blabla2}. Let also
$$\rho_n=\rho_{\CF_n}:\pi_1(\Sigma,*)\to\Diff[0,1]$$ 
be the holonomy of the $\CF_n$ and let $I_n\subset I_{n-1}$ be a maximal open subinterval such that $f(I_{n-1})\cap I_{n-1}=\emptyset$. We claim that, if we choose $f$ in \eqref{eq-diffeo} sufficiently close to the identity, then the pair $(\rho_n,I_n)$ satisfies the claims in Proposition \ref{prop-approx}. 

In fact, as we mentioned earlier, (5) is already satisfied. By choice of $I_{n-1}$, we are also satisfying (1). Then, if we choose $f$ in \eqref{eq-diffeo} $C^{n+10}$-close to $\Id$ then the associated foliation of $\CG$ of $\BS^1\times[-1,1]\times I_{n-1}$ is $C^n$-close to the trivial foliation by cylinders $\BS^1\times[-1,1]\times\{t\}$. This implies in turn that the foliation $\CH$ on $\BH^2/\Gamma_{n-1}\times I_{n-1}$ is again close to the trivial foliation and, surprise, surprise, this implies again that the perturbed foliation $\CF_n$ is close to the unperturbed foliation $\CF_{n-1}$. The off-shot of all this, is that as long as we choose $f$ sufficiently close to the identity, then (2) in Proposition \ref{prop-approx} is also satisfied.

To see that (3) and (4) are satisfied note that, by (4') above, it suffices to prove that they both hold if we restrict the holonomy representation $\rho_n=\rho_{\CF_n}$ to the subgroup $\Gamma_n$. Noting that $\Gamma_n$ is nothing but the image of $\pi_1(\phi(\BH^2/\Gamma_{n-1}\times I_{n-1}),\{*\}\times I_{n-1})$ under the restriction of the projection $\pi:M_{n-1}\to\Sigma$, we obtain from the discussion at the end of section \ref{sec-review} it follows that the holonomies of $\rho_n=\rho_{\CF_n}$ of $\CF_n$ and $\rho_\CH:\Gamma_n\to\Diff(\bar I_n)$ of the foliation $\CH$ on $\BH^2/\Gamma_{n-1}\times I_{n-1}$ are related by
$$\rho_n((\pi_*(\gamma))\vert_{I_n}=\rho_{\CH}(\gamma)$$
for $\gamma\in\Gamma_{n-1}=\pi_1(\phi(\BH^2/\Gamma_{n-1}\times I_{n-1}),\{*\}\times I_{n-1})$. The holonomy $\rho_\CH$ of the foliation $\CH$ is given by
\begin{equation}\label{eq-this is unreadable}
\rho_\CH(\gamma)=f^{\langle\gamma,c\rangle}
\end{equation}
which means that $\rho_\CH(\gamma)(I_n)\cap I_n=\emptyset$ unless $\langle\gamma,c\rangle=0$, that is unless $\gamma\in\Gamma_n$. It also implies that $\Ker(\rho_\CH)=\Gamma_n$, which means in particular that $\rho_\CH(\gamma)$ fixes $I_n$ pointwise if $\gamma\in\Gamma_n$. These observations, combined with \eqref{eq-this is unreadable} show that $\rho_n$ satisfies (3) and (4) from Proposition \ref{prop-approx}.  

This completes the induction step and thus the proof of Proposition \ref{prop-approx}.\qed
\medskip

\begin{proof}[End of the proof of Theorem \ref{main}]
It remains to show that the induced action on $[0,1]$ has only countably many points with non-trivial stabiliser, that it has no global fixed points in $(0,1)$, and that it is topologically transitive. 

Note that by points (3) and (5) of Proposition \ref{prop-approx}, the stabilizer of a point $x \in K_n \smallsetminus K_{n-1}$ is precisely $\Gamma_n$. Also, the stabilizer of any point $x \in [0,1] \smallsetminus K_\infty$ is trivial by point (4) of Proposition \ref{prop-approx} and the fact that $\bigcap \Gamma_n =\{ Id \}$. It follows that the points in $K_\infty= \cup_n K_n$ are the only points in $(0,1)$ with non-trivial stabiliser with respect to the action induced by the limiting representation. In our construction, $K_\infty$ is countable because we assume that $I_n$ is a maximal open subset of $I_{n-1}$ with $f(I_n)\cap I_n=\emptyset$. 

Noting now that $\Gamma_n$ is a proper subgroup of $\pi_1(\Sigma,*)$ for $n\ge 1$ we get that the only points on $[0,1]$ fixed by the whole group are those in $K_0=\{0,1\}$. In other words, the action on $(0,1)$ has no global fixed point.

Finally note that $\Gamma$ and $\Gamma/\Gamma_n$ act in the same way on the set $\pi_0([0,1] \smallsetminus K_n)$ of connected components of the complement of $K_n$. Since the latter group acts transitively by construction so does the former. Noting that the diameter of the connected components of $[0,1] \smallsetminus K_n$ tends to $0$, it follows that the orbit of every $x \notin K_\infty$ is dense. Hence, the action is topologically transitive.
\end{proof}

\end{document}